\documentclass[a4paper,10pt,twoside]{article}

\usepackage[body={175mm,250mm}]{geometry}
\usepackage{a1}
\usepackage[english]{babel}
\usepackage[latin1]{inputenc} 
\usepackage{amsfonts,amssymb}
\usepackage{amsmath}
\usepackage{graphicx}	

\usepackage{makeidx}
\makeindex

\def\mapright#1#2#3{\smash{\mathop{\hbox to
#3{\rightarrowfill}}\limits^{#1}_{#2}}}

\def\mapleft#1#2#3{\smash{\mathop{\hbox to
#3{\leftarrowfill}}\limits^{#1}_{#2}}}

\def\mapright#1#2{\smash{\mathop{\hbox to 0.90cm{\rightarrowfill}}\limits^{#1}_{#2}}}
\def\mapleft#1#2{\smash{\mathop{\hbox to 0.90cm{\leftarrowfill}}\limits^{#1}_{#2}}}

\def\mapleftright#1#2{\smash{\mathop{\hbox to 0.80cm{\leftarrowfill \rightarrowfill}}\limits^{#1}_{#2}}}

\title{A subtle new invariant for framed oriented knots and links} 
\author{Sóstenes Lins}

\date
{\small This work is dedicated to Louis Kauffman for his great
contributions to combinatorial knot theory.}

\begin{document}

\maketitle

\begin{abstract}

We produce a facial state sum on plane diagrams of a knot or a link which admits
an invariant specialization under Polyak's recent set of generating of 4
Reidemeister moves. Thus an isotopy invariant of framed links is obtained.
Each state is a complete coloring of the faces of the diagram
into white and black faces so that no two black faces share an edge. Each
state induces a monomial in a ring of 16 variables. The sum of the states,
properly specialized defines the new invariant. In despite of its simplicity
it complements Jones invariant in distinguishing mirror pairs of links. 
In particular it proves that $9_{42}$ is distinct from its mirror image.
For this pair of knots both the Jones Polynomial and Kauffman 2-variable polynomial
fail.

\end{abstract}

\section{Introduction}

This work was inspired by the combinatorial methods to study links pioneered by Kauffman in 
\cite{kauffman1987state,kauffman1990invariant}. 
Recently Polyak (\cite{polyak2009minimal,polyak2010minimal} has proven in a lucid paper that the 
4 moves below are enough to
obtain all 16 oriented forms of Reidemeister moves. There are 4 moves of type 1
(1a,1b,1c,1d), 4 of type 2 (2a,2b,2c,2d) and 8 of type 3 (3a,3b, 3c, 3d, 3e, 3f, 3g, 3h). 
\vspace{0.5cm}
\begin{figure}[!ht]
\begin{center}
\includegraphics[width=15cm]{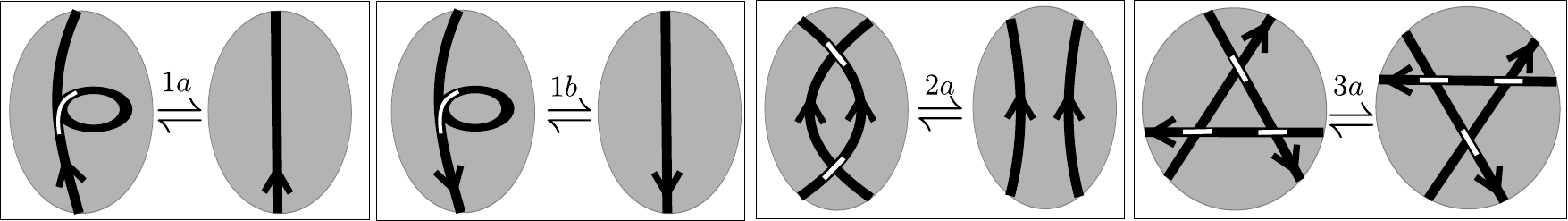} \\
\caption{\sf Polyak's set of 4 generating Reidemeister moves}
\label{fig:generating}
\end{center}
\end{figure}

This result is handy in the sense that to check for invariance we need only to care about
a substantially smaller subset of moves. In our setup not only the 2 above moves of type 1,
but the 4 of them are obtained for free. Therefore we are left only with moves 2a and 3a to
deal with. The approach yields a new invariant for framed links because the {\em ribbon move}
(see chapter 12, section 1 of \cite{kauffman1994tlr}), is satisfied. This invariant is simple
to define, subtle and fairly strong. It is particularly good at distinguishing pairs of links
not distinguished by the Jones invariant. The invariant is easier to compute when compared
with Jones' since it has substantially less number of states. Also it admits recursiveness and
full parallelism in its computation contrary to the computation of the Jones' polynomial. 
The issue of effectively
computing link invariants is important in various areas of applied science. A typical example
of application to computational molecular biology is \cite{emmert2006algorithmic}. In these applied areas,
it might be important to compute good invariants
for links with more crossings than attainable by current invariants.

For checking the performance of the new invariant as compared with the Jones invariant
we have produced the following experiment in the 2978 (including the unknot) knots up to 12 crossings:
compute both invariants in each mirror pair of knots and report the discrepancies.
Out of the 2978 pairs only 21 behave differently.
Each discrepancy means that either the pair is distinguished by Jones and not by the new
invariant (2 $\times$ 1 to Jones), or else $2 \times 1$ for the new invariant.
The total sum of scores on the 21 discrepancies 
is 30 $\times$ 12 in favor of the new invariant.
Its first victory is precisely $9_{42}$.

I am indebted to the Centro de Inform\'atica, UFPE/Recife, Brazil  
for financial support. I am also supported by a research 
grant from CNPq/Brazil, Proc. 301233/2009-8. I want to thank Peter Johnson,
from DMAT-UFPE for helpful conversations about the theme of this work and
for calling my attention to Polyak's paper, \cite{polyak2009minimal}.

\section{The states, the variables and the state sum}
By a {\em framed link} we mean the union of 
$n$ disjoint  copies of  $[\-\epsilon,\epsilon] \times \mathbb{S}^1$ embedded into $\mathbb{R}^3$. 
A knot is also a link (case $n=1$). A framed link is conveniently described by its special projections
into a plane $\mathbb{R}^2$. We make sure that the crossings of the projection when $\epsilon \rightarrow 0$
are transversals and there are no triple points. The image of the projection becomes a 4-regular
finite plane graph $\mathcal{D}$, where each vertex is a crossing. At each crossing we need to know which 
one of the two strands is up. In our figures we let the projection be thick black lines so that
at each crossing a thinner white segment indicates the strand which is above. We impose a further
restriction on our projections: the image $\mathcal{D}$ must be a {\bf connected} graph. This implies that each
{\em face}, that is a connected component of the difference $\mathbb{R}^2 \backslash \mathcal{D}$,
is either an open disk or the complement of a disk (case of the external face).
A final restriction is that the projection has to have at least one crossing. Therefore
to compute the invariant of the unknot as a circle, we have to replace the circle, say, by the unknot
with two crossings. A projection
satisfying all the listed restrictions above is named an {\em adequate projection}. It is fairly easy
to modify, by isotopies and Reidemeister moves of type 2, an arbitrary projection free of triple points and
tangencies to obtain an adequate projection. This process is named a {\em preparation} 
for the projection. It is not difficult to prove that any two preparations of the same link have the 
same evaluation under our scheme. This is proved in an adequate generality in a forthcoming joint paper with
P. Johnson, \cite{joli2012A}. 

Our links have all their components oriented and so at each crossing there is a 
distinguished angle formed by two incoming segments. For the crossing $X$ let the 4 angles
be encompassed by faces $F_0^\chi, F_1^\chi, F_2^\chi, F_3^\chi$, where   $F_0^X$ is the face encompassing
the distinguished angle and the other proceed in a counterclockwise fashion. These 4 faces need
not to be distinct. Given a link diagram $\mathcal{D}$, let $\mathcal{F}$ be the set of faces
of $\mathcal{D}$. A {\em state} of $\mathcal{D}$ is a function $\sigma: \mathcal{F} \longrightarrow \{0,1\}$
with the restriction that two faces sharing an edge of $\mathcal{D}$ cannot both have 1 as their
images under $\sigma$. In the diagrams a state is represented by its image displayed as black and
white spots inside the faces they represent. Given a state $\sigma$ and a crossing $X$, 
let $\sigma_\chi$ denote the sequence of 4 bits which are the 
image of the faces $F_0^\chi, F_1^\chi, F_2^\chi, F_3^\chi$ under $\sigma$.
Denote by $|\sigma_\chi|$ the decimal value of $\sigma_\chi$. A state induces a set of variables
at each crossing $\chi$: either $x_{|\sigma_\chi|}$, if $\chi$ is a negative crossing or else
$X_{|\sigma_\chi|}$, if it is positive. There are 14 possible crossing variables. They are:
$x_0, X_0, x_1, X_1, x_2, X_2, x_4, X_4, x_5, X_5, x_8, X_8, x_{10}, X_{10}$. The state also induces
a set of variables at each face: a white spot corresponds to $f_0$ and a black spot corresponds to $f_1$.
The face variables $f_0$ and $f_1$ are added to the set of 14 crossing variables defining the
16 variables of our state sum model. The {\em monomial} induced by a 
state is the product of the variables at the crossings together
with the product of the variables at the faces. The {\em state sum} of $\mathcal{D}$, denoted 
$\left[\mathcal{D}\right]$,  is the sum of all the monomials induced by all the states. 

\begin{theorem}
Consider the following set of 25 polynomial equations in the 16 variables.

\begin{center}
\begin{tabular}{|l|r|}
\hline
${f_0}^3 {x_0} {X_0}+{f_0}^2 {f_1} {X_2} {x_8}={f_0}, $ &
${f_0}^2 {f_1} {x_0} {X_8}+{f_0} {f_1}^2  {x_8} {X_{10}}=0$\\
${f_0}^3 {x_1} {X_1}={f_0}$ &
${f_0}^2 {f_1} {X_0} {x_2}+{f_0} {f_1}^2  {X_2} {x_{10}}=0$\\
${f_0} {f_1}^2 {x_2} {X_8}+{f_1}^3 {x_{10}} {X_{10}}-{f_1}$&
${f_0}^3 {x_4} {X_4}={f_0}$\\
${f_0}^3 {x_5} {X_5}={f_0}$ & ------------ \\ \hline \hline

$f_0x_0X_0^2+f_1 x_1 X_4^2 = f_0 x_0 X_0^2+f_1 x_1 X_1^2$&
${f_0} {x_0} {X_0} {X_1} + {f_1} {x_1} {X_4} {X_5}={f_0} {X_0} {X_2} {x_8}$\\
${f_0} {X_0} {X_2} {x_8}={f_0} {x_0} {X_0} {X_4}+{f_1} {x_1} {X_1} {X_5}$&
${f_0} {X_0}^2 {x_1}+{f_1} {X_4}^2 {x_5}={f_0} {x_0} {X_2} {X_8},$\\
${f_0} {X_0} {x_1} {X_1}+{f_1} {X_4} {x_5} {X_5}={f_0} {X_2}  {x_8} {X_{10}}$&
${f_0} {X_0} {x_2} {X_8}={f_0} {x_0} {X_0} {X_4}+{f_1} {x_1} {X_1} {X_5}$\\
$f_0 X_1 x_ 2 X_8 = f_0 X_2 X_4 x_8$&
$f_0 X_2 X_8 x_{10} =f_0 x_0 X_4^2 + f_1 x_1 X_5^2$\\
$ f_0 x_0 X_0 X_ 1 + f_1 x_4 X_4 X_5 = f_0 X_0 x_2 X_8 $&
$ f_0 x_0 X_1 ^2 + f_1 x_4 X_5^2 = f_0 X_2 X_8 x_{10}$\\
$f_0 X_1 X_2 x_8 = f_0 x_2 X_4 X_8$&
$ f_0 X_0 x_1 X_1 + f_1 X_4 x_5 X_5 = f_0 x_2 X_8 X_{10}$\\
$f_0 x_1 X_1^2 + f_1 x_5 X_5 ^2 = f_0 x_{10} X_{10}^2$&
$f_0 x_0 X_2 X_8 = f_0 X_0^2 x_4 + f_1 X_1^2 x_5$\\
$ f_0 X_2 x_8 X_{10} = f_0 X_0 x_4 X_4 + f_1 X_1 x_5 X_5$&
$ f_0 x_1 X_2 X_8 = f_0 X_2 x_4 X_8 $\\
$ f_0 x_2 X_8 X_{10}=f_0 X_0 x_4 X_4 + f_1 X_1 x_5 X_5$&
$f_0 x_{10} X_{10}^2 = f_0 x_4 X_4^2 + f_1 x_5 X_5^2$\\ \hline
\end{tabular} 
\end{center}
If the 16 variables are specialized so that these equations are satisfied, then
the state sum $\left[\mathcal{D}\right]$ is invariant under moves 2a and 3a.
\label{theo:invars23}
\end{theorem}

\begin{proof}
The proof is given in Figs. \ref{fig:move2a} and \ref{fig:move3a}. The equations are
easily obtained once each possible color configuration of the spots in the boundary faces is fixed.
\end{proof}

\vspace{0.5cm}
\begin{figure}[!ht]
\begin{center}
\includegraphics[width=16cm]{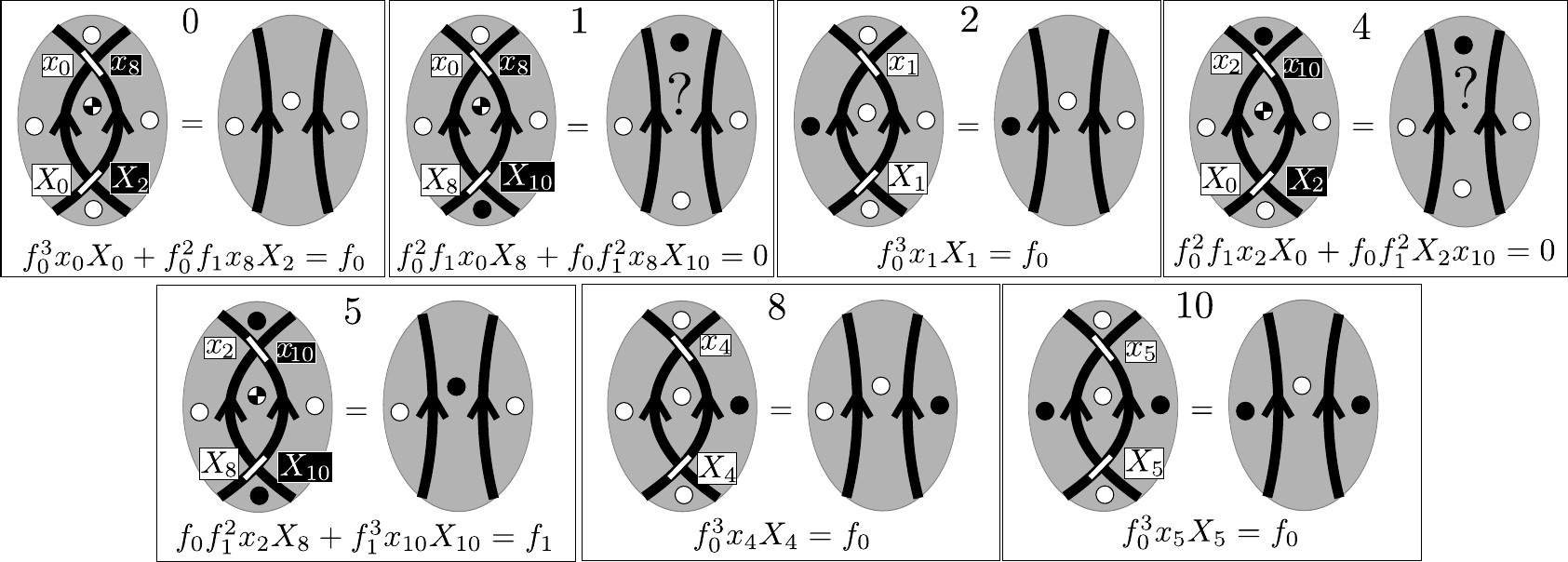} \\
\caption{\sf Forcing invariance under a restricted form of Polyak's Reidemeister move 2a: 
this move is non-local with respect to the facial
state sum employed. Indeed it only applies in the case that the top 
and bottom faces (incident to the central bivalent face in the lefthand side diagrams) are distinct. 
Therefore three distinct faces coalesce into one in the righthand side diagrams. 
With this care we have complete 
information about how many faces are involved to adjust the polynomial equations accordingly. Variables
$f_0$ and $f_1$ are supposed to be invertible and so we do not need to put in the equations 
the face variables corresponding
to the left and right faces, which are maintained in both sides of each pair of diagrams 
and so cancell each other. When the top and
bottom faces of a left diagram receive distinct marks, then the corresponding right side do not exist,
and so must be evaluated as zero (boundary cases 1 and 4). 
Note that of the 16 possible boundary configurations
of black/white spots only the above 7 have no adjacent black spots. }
\label{fig:move2a}
\end{center}
\end{figure}

\vspace{0.5cm}
\begin{figure}[!ht]
\begin{center}
\includegraphics[width=17cm]{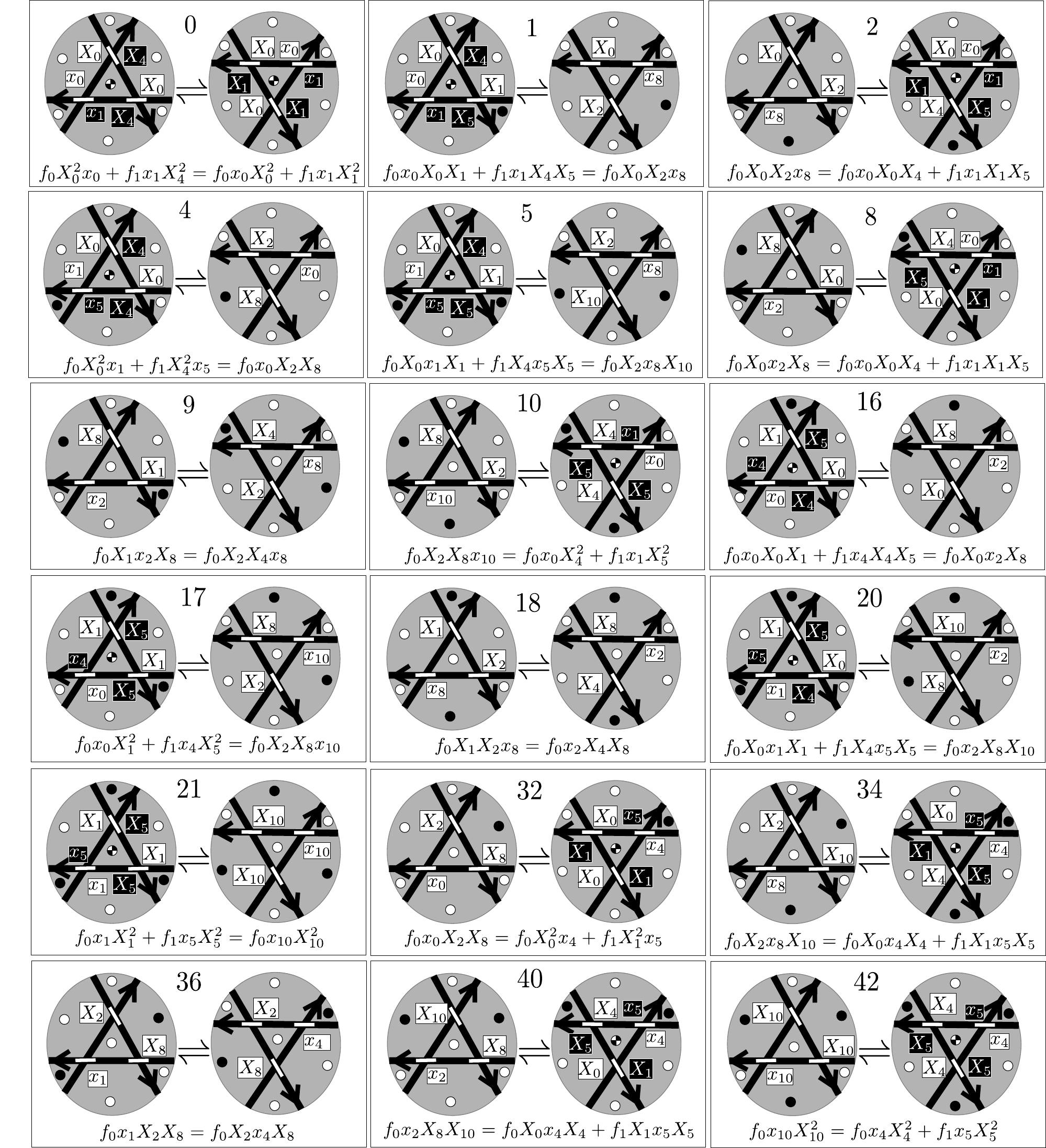} \\
\caption{\sf Forcing invariance under Polyak's Reidemeister move 3a: in this case
the move is entirely local. Equations implying invariance are obtained with no aditional
constraints. Out of the 64 boundary configurations for this move, only 18 have no adjacent
pair of black spots. These are displayed above and each one contributes with an equation which
must be satisfied for the invariance.}
\label{fig:move3a}
\end{center}
\end{figure}

\noindent
For the basics about rings, ideals and Groebner basis see \cite{cox2007ideals}.
If $p \in \mathcal{R}$ and $\mathcal{I}$ is a fixed ideal of $\mathcal{R}$ , 
let $\widehat{p} := p+\mathcal{I} \in  \mathcal{R}/ \mathcal{I}$. In this
quotient ring $\widehat{p}=\widehat{q} \Leftrightarrow p-q \in \mathcal{I}$ and
$\widehat{p\, q}=\widehat{p}\, \widehat{q}$.

\begin{proposition}
Consider the ring $\mathcal{R}=\langle s, S \rangle$. Define
$p_1= 1 + s + s^2 + s^3 + s^4$, 
$p_2= 1 + S + S^2 + S^3 + S^4$, 
$p_3 = s^4 + s^3 S + s^2 S^2 + s S^3 + S^4$, 
Let $\mathcal{I}$ be the ideal $\mathcal{I}=\langle p1, p2, p3 \rangle.$ 
The following assignment 
\begin{center}
$
\begin{array}{cccc}
f_0 := \widehat{s}^2\widehat{S}^2 & 
f_1 := \widehat{s}^4+\widehat{S}^4 &
x_0 := \widehat{s}+\widehat{S} & 
X_0 = \widehat{s}+ \widehat{S} \\
x_1 := \widehat{S} & X_1 := \widehat{s} & 
x_{10} := -\widehat{s}^4 \widehat{S}^2 & 
X_{10} := -\widehat{s}^2 \widehat{S}^4 \\
x_2 := \widehat{s} & X_2 :=  \widehat{S} & x_4 := \widehat{S} & X_4 := \widehat{s} \\
x_5 := -\widehat{s}^2 \widehat{S}^4 & X_5 := -\widehat{s}^4 \widehat{S}^2 & 
x_8 := \widehat{s} & X_8 := \widehat{S}\\
\end{array}
$
\end{center}
satisfies all the 25 equations for the invariance of 2a and 3a
in the quotient ring $\mathcal{R}/\mathcal{I}$.
\label{prop:assig}
\end{proposition}
\begin{proof}
 Routine check.
\end{proof}

The above assignment of variables in the ring $\mathcal{R}$ was first obtained
from Groebner basis computations with thousands of generators including all 12 
forms of the moves 2 and 3 in a much more general experiment which included beyond orientation
checkerboard shading and no 1-1 restriction on adjacent faces. This condition was naturally suggested
by the algebra. By imposing it an adequate small GB is obtained. The results show that shading
is irrelevant. By not imposing the 1-1 restriction produces a stronger invariant. 
But the algebra is messy and further research needs to be done to clarify the situation. 

To check the invariance of framed links relative to  
our assignment in Proposition \ref{prop:assig} is enough, by Polyak's result, 
to take care of moves 1 and to check the 25 equations. Polyak's paper (\cite{polyak2010minimal})
permits a proof of invariance which does not need computers and is
fully presented here.

\section{Dealing with Reidemeister moves of type 1}

The 4 forms of move 1 
are dealt with by an easy compensation
involving the total writhe of $\mathcal{D}$ in the spirit of Kauffman's original bracket, 
\cite{kauffman1987state}.  The proof of their invariance depends on our specific bracket evaluation
on $\mathcal{R}/\mathcal{I}$.

\begin{lemma}
Assume we are working in a quotient ring satifying
$$
f_0 x_0 + f_1 x_4 = f_0 x_{10} = f_0 x_1 + f_1 x_5, \hspace{4mm}
f_0 X_0 + f_1 X_4 = f_0 X_{10} = f_0 X_1 + f_1 X_5.
$$
That is the case of our case $\mathcal{R}/\mathcal{I}$.
Then, the cancellation of the 4 types of curls have the following multiplicative
effects on the evaluation of a diagram:
\begin{center}
\includegraphics[width=17cm]{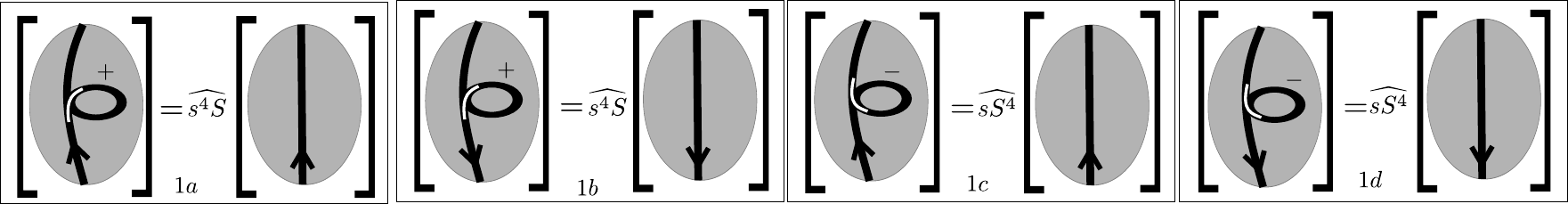} \\
\label{lem:moves1}
\end{center}
\end{lemma}

\begin{proof}
Consider the 3 possibilities for the boundary of move 1a.
\begin{center}
\includegraphics[width=13cm]{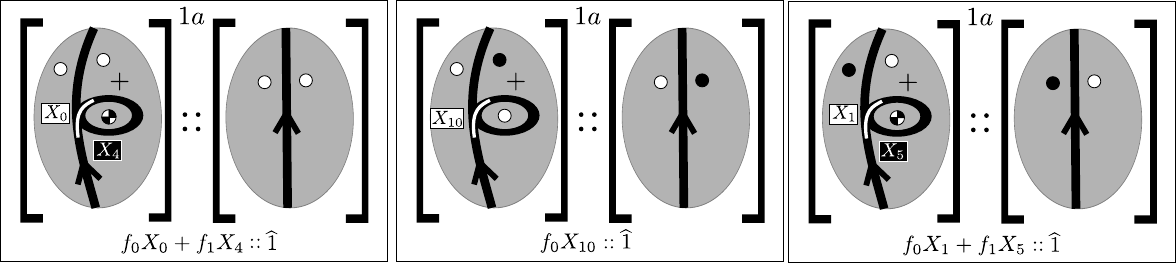}\ .
\end{center}
The proof for this move follows from the fact that in $\mathcal R / \mathcal{I}$
we have $f_0X_0+f_1X_4 = f_0 X_{10} = 
f_0 X_1 + f_1 X_5 = \widehat{s}^4 \widehat{S}$. Thus, the factor is constant,
independent of the colors of the spots in the boundary face. The proof for move 1b is exactly
the same, since only the sign and not the winding number of the curl is relevant for
our evaluation. The proof of 1c and 1d are similar. The three boundary configurations 
for the case 1c appear below. Case 1d has exactly the same proof as 1c.
\begin{center}
\includegraphics[width=13cm]{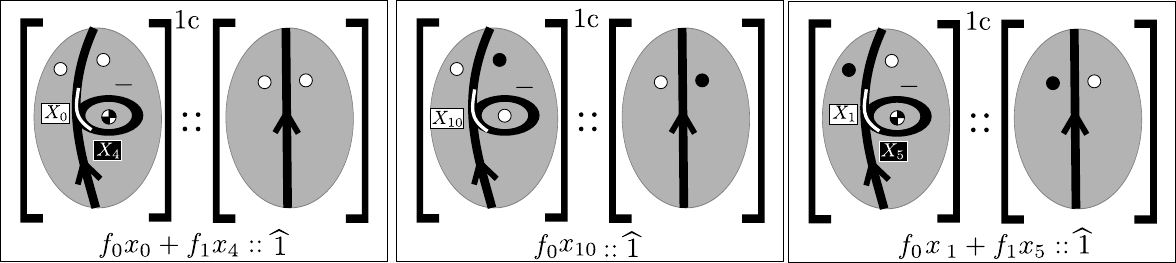}\ .
\end{center}
This time we have in $\mathcal R / \mathcal{I}$:  
$f_0x_0+f_1x_4 = f_0 x_{10} = f_0 x_1 + f_1 x_5 = \widehat{s}\widehat{S}^4$.
\end{proof}

\begin{definition}
Given a link diagram $\mathcal{D}$, let wr($\mathcal{D}$) denote the total writhe of $\mathcal{D}$.
Then define $ \lceil \mathcal{D} \rceil := 
(\widehat{s} \widehat{S}^4)^{wr(\mathcal{D})} \left[\mathcal{D}\right]$,
if $wr(\mathcal{D})\ge 0$ and
$ \lceil \mathcal{D} \rceil := 
(\widehat{s}^4 \widehat{S})^{wr(\mathcal{D})} \left[\mathcal{D}\right]$, 
if $wr(\mathcal{D})<0$.
\end{definition}

\begin{theorem}
 The state sum $\lceil \mathcal{D} \rceil$ is invariant under the 6 generating Reidemeister moves
1a, 1b, 1c, 1d, 2a 3a. Therefore it is an invariant of ambient isotopy for links.
\end{theorem}
\begin{proof}
The proof follows from Lemma \ref{lem:moves1}, which takes care of moves of type 1 and from
Theorem \ref{theo:invars23} since moves 2a and 3a do not change the writhe.
\label{theo:fullinvar}
\end{proof}

\begin{proposition}
  The state sum $\lceil \mathcal{D} \rceil$ is invariant under the ribbon move.
So it is an invariant of framed links.
\end{proposition}
\begin{proof}
 Straightforward.
\end{proof}

Given a total order in the set of monomials of a polynomial ring,
we can {\em reduce a polynomial relative to an ordered set of generators
for an ideal of the ring}. This produces a well determined {\em reduced polynomial}.
If the generating set is a Groebner basis, then the reduced polynomial
is canonical: that is, if $\widehat{p} = \widehat{q}$ and $p, q$ are
reduced polynomials relative to an ordered Groebner basis, then $p=q$
in the original ring $\mathcal{R}$. See a proof in \cite{cox2007ideals}. 
For our purposes, this is the crucial property of a Groebner basis. 
One such basis with two polynomials relative to the lexicographic
monomial ordering and variable ordering $(s,S)$ for our ideal $\mathcal{I}$ is
$$ \mathcal{B}=( S^4+S^3+S^2+S+1,s^3+s^2 S+s^2+s S^2+s S+s+S^3+S^2+S+1),
\hspace{3mm} \mathcal{I} = \langle \mathcal{B} \rangle.$$
Unfortunately the canonical forms of the polynomials associated to $\mathcal{B}$
lose the nice pentagonal symmetry inherent in the algebra. This loss of symmetry
makes impossible to display easily $ \lceil \mathcal{D^\star}\rceil$ 
in terms of  $ \lceil \mathcal{D}\rceil$ where  $\mathcal{D^\star}$ is the mirror of a given 
link $\mathcal{D}$. Next section fixes the situation.

\section{A numerical invariant and its comparison with the Jones invariant}

To make our invariant independent of canonical forms associated to Groebner bases,
we specialize further the definition of $\lceil \mathcal{D} \rceil$. Take the ring
$\mathcal{R}$ to be $\mathbb{Z}[e^{(2 \pi I / 5}]$, by defining $s:=e^{2 \pi I / 5}$ and
$S:=e^{8 \pi I / 5}$. Note that with this assignment the three generators of
$\mathcal{I}$ become zero. Under this scope, 
$\lceil \mathcal{D} \rceil \in \mathbb{Z}[e^{2 \pi I / 5}]$, for any link diagram $\mathcal{D}$
and we can drop all hats of the previous section. In practice this specialization
does not weaken the invariant: the number of 3679 distinct values of the invariant 
for the 2978 pairs of knots up to 12 crossings is maintained,
with and without specializations.

\begin{proposition}
An element $z \in \mathbb{Z}[\exp(2 \pi I / 5]$ can be uniquely written as 
$ z = n\ e^{2 \pi I / 5} + o\  e^{4 \pi I / 5} + p\  e^{6 \pi I / 5}
+ q\ e^{8 \pi I / 5}$, for integers $n,o,p,q$.
\end{proposition}
\begin{proof}
Let a generic element be
$ z = m + n'\ e^{2 \pi I / 5} + o'\  e^{4 \pi I / 5} + p'\  e^{6 \pi I / 5} + q'\ e^{8 \pi I / 5}$.
Then, defining $n=n'-m,\ o=o'-m,\ p=p'-m,\ q=q'-m$ the proposition follows.
\end{proof}

\begin{definition}
The {\em canonical form} of $ z = n\ e^{2 \pi I / 5} + o\  e^{4 \pi I / 5} + p\  e^{6 \pi I / 5}
+ q\ e^{8 \pi I / 5}$ is denoted by $ \lfloor n,o,p,q \rfloor $.
\end{definition}

\begin{proposition}
Let the canonical form of $\lceil \mathcal{D} \rceil$ be $\lfloor n,o,p,q\rfloor$.
If $ \mathcal{D}$ and its mirror image
are ambient isotopic, then $q=n$ and $p=o$.
\end{proposition}
\begin{proof}
 It follows from the fact that the invariant of the mirror of a link diagram is
obtained by interchanging $s$ and $S$, therefore, by taking the conjugate of the
complex number. In terms of the canonical form, by taking the reverse or the quadruple.
\end{proof}

\subsection{Table of discrepances between $\lceil \mathcal{D}\rceil$ and
\text{\bf{Jones$(\mathcal{D})$}} for knots up to 12 crossings}

Except by our implementation in Mathematica (\cite{wolfram2012mathematica}) for
obtaining the canonical form of $\lceil \mathcal{D}\rceil$, the data for the
table below is taken from \cite{bar2005knot} and from \cite{cha2009knotinfo}. 
The 15 entries in boldface corresponds to the knots $\mathcal{D}$ distinguished by $\lceil \mathcal{D}\rceil$
and not by the Jones$(\mathcal{D})$. In the other 6 entries the reverse occurs.
For all other 2957 (including the unknot) knots up to 12 crossings 
the behavior of the two invariants coincide: both fail or both succeed.

\begin{center}
$
\begin{array}{rrrr}

& & &\text{\bf{Canonical form}}\\

\text{\bf{Rank}} &\text{\bf{$\mathcal{D}$-name}} & 
\text{\bf{Jones$(\mathcal{D})$}}&\text{\bf form of $\lceil \mathcal{D}\rceil$}\\

4&5_1&  -q^7+q^6-q^5+q^4+q^2& \lfloor -4, -1, -1, -4\rfloor \\

78&\bf{9_{42}} & q^3+\frac{1}{q^3}-q^2-\frac{1}{q^2}+q+\frac{1}{q}-1&\lfloor4, 0, 3, 3\rfloor\\

209&10_{124}& -q^{10}+q^6+q^4&\lfloor -4, -1, -1, -4\rfloor \\

224&10_{139} & -q^{12}+q^{11}-q^{10}+q^9-q^8+q^6+q^4&\lfloor -9, -6, -6, -9 \rfloor \\

237&10_{152} & q^{13} -2 q^{12}+2 q^{11}-3 q^{10}+2 q^9-2 
q^8+q^7+q^6+q^4&\lfloor-14, -6, -6, -14\rfloor \\

636&\bf{K11n19} & q^2-q+1-\frac{1}{q}+\frac{1}{q^2}&\lfloor2, -4, 2, 0\rfloor\\

641&\bf{K11n24} & -q^4-\frac{1}{q^4}+2 q^3+\frac{2}{q^3}-3
   q^2-\frac{3}{q^2}+4 q+\frac{4}{q}-3&\lfloor14, 0, 8, 8\rfloor\\

699&\bf{K11n82} & -q^4-\frac{1}{q^4}+2 q^3+\frac{2}{q^3}-2 q^2-\frac{2}{q^2}+3
   q+\frac{3}{q}-3&\lfloor9, 0, 8, 8\rfloor\\

756&K11n139 & q^8-q^7+q^6-2 q^5+q^4-q^3+q^2+1&\lfloor-4, -1, -1, -4\rfloor\\

1471&\bf{K12a0669} & -q^6-\frac{1}{q^6}+2 q^5+\frac{2}{q^5}-4 q^4-\frac{4}{q^4}+6
   q^3+\frac{6}{q^3}-7 q^2-\frac{7}{q^2}+9 q+\frac{9}{q}-9&\lfloor0, -13, 6, -13\rfloor\\

1973&\bf{K12a1171} & -q^6-\frac{1}{q^6}+3 q^5+\frac{3}{q^5}-6 q^4-\frac{6}{q^4}+10
   q^3+\frac{10}{q^3}-13 q^2-\frac{13}{q^2}+16 q+\frac{16}{q}-17&\lfloor 0, -28, 16, -28\rfloor\\

1981&\bf{K12a1179}& -q^6-\frac{1}{q^6}+3 q^5+\frac{3}{q^5}-5 q^4-\frac{5}{q^4}+8
   q^3+\frac{8}{q^3}-10 q^2-\frac{10}{q^2}+12 q+\frac{12}{q}-13&\lfloor0, -18, 11, -18\rfloor\\

2007&\bf{K12a1205} & -q^6-\frac{1}{q^6}+3 q^5+\frac{3}{q^5}-6 q^4-\frac{6}{q^4}+10
   q^3+\frac{10}{q^3}-14 q^2-\frac{14}{q^2}+17 q+\frac{17}{q}-17&\lfloor 0, -28, 21, -28\rfloor\\

2239&K12n0149 & -q^{12}+q^{11}-q^{10}+q^9-q^7+2 q^6-2 q^5+2 q^4-q^3+q^2&\lfloor -14, -6, -6, -14\rfloor\\

2340&\bf{K12n0250} & -q^5-\frac{1}{q^5}+2 q^4+\frac{2}{q^4}-3 q^3-\frac{3}{q^3}+5
   q^2+\frac{5}{q^2}-6 q-\frac{6}{q}+7&\lfloor 12, -9, 12, 0\rfloor\\

2368&\bf{K12n0278} & -q^5-\frac{1}{q^5}+3 q^4+\frac{3}{q^4}-5 q^3-\frac{5}{q^3}+8
   q^2+\frac{8}{q^2}-10 q-\frac{10}{q}+11&\lfloor 22, -14, 22, 0\rfloor\\

2452&\bf{K12n0362} & -q^4-\frac{1}{q^4}+3 q^3+\frac{3}{q^3}-4 q^2-\frac{4}{q^2}+5
   q+\frac{5}{q}-5 &\lfloor 0, -13, 6, -13\rfloor\\

2580&\bf{K12n0490} & -q^5-\frac{1}{q^5}+3 q^4+\frac{3}{q^4}-6 q^3-\frac{6}{q^3}+10
   q^2+\frac{10}{q^2}-12 q-\frac{12}{q}+13 &\lfloor 27,-19,27,0\rfloor \\

2676&\bf{K12n0586} & -q^5-\frac{1}{q^5}+4 q^4+\frac{4}{q^4}-8
   q^3-\frac{8}{q^3}+13 q^2+\frac{13}{q^2}-16 q-\frac{16}{q}+17 & \lfloor 37, -24, 37, 0 \rfloor\\

2700&\bf{K12n0610} & -q^5-\frac{1}{q^5}+3 q^4+\frac{3}{q^4}-5 q^3-\frac{5}{q^3}+8
   q^2+\frac{8}{q^2}-10 q-\frac{10}{q}+11 &\lfloor 22, -14, 22, 0\rfloor\\

2785&\bf{K12n0695} & -q^5-\frac{1}{q^5}+3 q^4+\frac{3}{q^4}-7 q^3-\frac{7}{q^3}+11
   q^2+\frac{11}{q^2}-13 q-\frac{13}{q}+15 &\lfloor 32, -19, 32, 0\rfloor\\

\end{array}
$
\end{center}

The {\bf $\mathcal{D}$-name} of each knot up to 10 crossings is from Rolfsen's table \cite{rolfsen2003knots}.
For 11 and 12 crossings we use the DT-name of KnotInfo \cite{cha2009knotinfo}. 
{\bf Rank} is the {\em name rank}
of this same reference, starting with rank 1 for the unknot. 
Because of the simplicity of getting the expression of both invariants for $\mathcal{D}^\star$ from
the invariants for $\mathcal{D}$,
each entry is in fact a pair of mirror knots.
However, the order of the pairs for the two invariants may not be the same.

\section{Conclusion}
We have defined a new ambient isotopy for links based on a state sum on the faces of
the link diagram. The number of valid states is substantially smaller when compared with
the $2^n$ states for Jones invariant. By the simplicity of its definition and the fact that
it is not implied by the Jones Polynomial not even by the Kauffman 2-variable polynomial, this new
invariant deserves to be further investigated. A particularly wellcome behavior that it possess
is the following: it distinguishes each pair of knots with the same 
Jones polynomial and which have distinct minimum crossing number. This is an empirical fact
which holds in  the realm of knots up to 12 crossings. Another important aspect
is the possibility of using recursion and parallellism in its computation. We believe that links
up to 100 crossings will be amenable to treat under a proper recursive/parallel implementation.


\bibliographystyle{plain}
\bibliography{bibtexIndex.bib}

\begin{thebibliography}{10}

\bibitem{bar2005knot}
D.~Bar-Natan, S.~Morrison, et~al.
\newblock The knot atlas.
\newblock {\em Web page}, 2005.

\bibitem{cha2009knotinfo}
J.C. Cha and C.~Livingston.
\newblock Knotinfo: Table of knot invariants.
\newblock {\em Web page}, 2009.

\bibitem{cox2007ideals}
D.~Cox and J.~Little.
\newblock {\em Ideals, varieties \& algorithms: An introduction to
  computational {A}lgebraic {G}eometry \& {C}ommutative {A}lgebra
  ({U}ndergraduate {T}exts in {M}athematics)}.
\newblock Springer-Verlag, 2 edition, 2005.

\bibitem{emmert2006algorithmic}
F.~Emmert-Streib.
\newblock Algorithmic computation of knot polynomials of secondary structure
  elements of proteins.
\newblock {\em Journal of Computational Biology}, 13(8):1503--1512, 2006.

\bibitem{joli2012A}
P.~Johnson and S.~Lins.
\newblock A graphical calculus for tangles on surfaces.
\newblock {\em In preparation}, 2012.

\bibitem{kauffman1987state}
L.H. Kauffman.
\newblock State models and the {J}ones polynomial.
\newblock {\em Topology}, 26(3):395--407, 1987.

\bibitem{kauffman1990invariant}
L.H. Kauffman.
\newblock An invariant of regular isotopy.
\newblock {\em Trans. Amer. Math. Soc}, 318(2), 1990.

\bibitem{kauffman1994tlr}
L.H. Kauffman and S.~Lins.
\newblock {Temperley-Lieb Recoupling Theory and Invariants of 3-manifolds}.
\newblock {\em Annals of Mathematical Studies, Princeton University Press},
  134:1--296, 1994.

\bibitem{polyak2009minimal}
M.~Polyak.
\newblock Minimal generating sets of {R}eidemeister moves.
\newblock {\em Arxiv preprint arXiv:0908.3127}, 2009.

\bibitem{polyak2010minimal}
M.~Polyak.
\newblock Minimal generating sets of {R}eidemeister moves.
\newblock {\em Quantum Topology}, 1:399--411, 2010.

\bibitem{rolfsen2003knots}
D.~Rolfsen.
\newblock {\em Knots and links}.
\newblock American Mathematical Society, 2003.

\bibitem{wolfram2012mathematica}
S.~Wolfram.
\newblock Mathematica. {W}olfram {R}esearch, version 8.
\newblock {\em {S}ymbolic {M}anipulation {S}oftware}, 2012.

\end{thebibliography}

\end{document}